\documentclass[journal,twoside,web]{ieeecolor}

\newtheorem{theorem}{Theorem}[section]
\newtheorem{lemma}[theorem]{Lemma}
\newtheorem{definition}[theorem]{Definition}

\newtheorem{proposition}[theorem]{Proposition}
\newtheorem{problem}{Problem}
\newtheorem{remark}[theorem]{Remark}
\newtheorem{assumption}{Assumption}
\newtheorem{example}[theorem]{Example}

\newcommand{\mc}{\mathcal}

\newcommand{\im}{\operatorname{Im}}

\newcommand{\rank}[1]{\operatorname{rank} \left( #1 \right)}
\newcommand{\spec}{\sigma}

\newcommand{\real}{\mathbb{R}} 

\newcommand{\naturalpos}{\mathbb{N}_{>0}}

\newcommand{\complex}{\mathbb{C}}
\newcommand{\tsp}{\mathsf{T}} 
\newcommand{\pinv}{\dagger} 
\newcommand{\inv}{{\negat 1}} 
\newcommand{\negat}{\scalebox{0.75}[.9]{\( - \)}}
\newcommand*{\QEDB}{\hfill\ensuremath{\square}}
\newcommand*{\QEDBL}{\hfill\ensuremath{\blacksquare}}
\newcommand{\diag}{\operatorname{diag}} 
  
\newcommand{\bA}{\mathbf{A}}  
\newcommand{\lmax}[1]{\sbs{\lambda}{max}\left(#1\right)}  
\renewcommand{\ast}{\text{\ding{91}}}

\newcommand{\map}[3]{#1: #2 \rightarrow #3}
\newcommand{\setdef}[2]{\{#1 \; : \; #2\}}
\newcommand{\sbs}[2]{{#1}_{\textup{#2}}}

\newcommand{\until}[1]{\{1,\dots,#1\}}
\newcommand{\norm}[1]{\Vert #1 \Vert}

\DeclareMathAlphabet{\mymathbb}{U}{BOONDOX-ds}{m}{n}
\newcommand{\one}{\mathds{1}}
\newcommand{\zero}{\mymathbb{0}} 
 
\usepackage[figurename=Fig.,font=footnotesize]{caption}

\usepackage{hyperref}
\usepackage{graphicx,color}
\graphicspath{{./IMG/}{images/},{./IMG-BIOS//},{../IMG/}}
\usepackage{amsmath}
\usepackage{amssymb}
\usepackage{mathrsfs}
\usepackage{subfigure}
\usepackage{url}
\usepackage{booktabs}
\usepackage{array}
\usepackage[table]{xcolor}
\usepackage{mathtools} 		
\usepackage{epstopdf}
\usepackage{cite}
\usepackage{algorithm}
\usepackage{algorithmic}
\usepackage{upgreek}
\usepackage{dsfont}
\usepackage{generic}
\usepackage{pifont} 

\def\BibTeX{{\rm B\kern-.05em{\sc i\kern-.025em b}\kern-.08em
    T\kern-.1667em\lower.7ex\hbox{E}\kern-.125emX}}
\title{$k$-dimensional Agreement in Multi-agent~Systems}

\begin{document}

\author{Gianluca Bianchin, Miguel Vaquero, Jorge Cort\'{e}s, and  Emiliano Dall'Anese\thanks{
G. Bianchin is with the ICTEAM Institute and Dept. of Mathematical Engineering at the University of Louvain. 
M. Vaquero is with the School of Science and Technology, IE University. J. Cort\'es is with the Department of Mechanical and Aerospace Engineering, University of California San Diego. 
E. Dall'Anese is with the Department of Electrical and Computer Engineering, Boston University.
Corresponding author: G. Bianchin. 
Email: \texttt{gianluca.bianchin@uclouvain.be}.
} \hspace{-1cm}}

\maketitle

\begin{abstract}
Given a network of agents, we study the problem of designing a distributed 
algorithm that computes $k$ independent weighted means of the network's initial 
conditions (namely, the agents agree on a $k$-dimensional space). 
Akin to average consensus, this problem finds applications in distributed 
computing and sensing, where agents seek to simultaneously evaluate $k$ 
independent functions at a common point by running a single coordination 
algorithm.
We show that linear algorithms can agree on quantities that are oblique 
projections of the vector of initial conditions, and we provide techniques to 
design protocols that are compatible with a pre-specified communication graph. 
More broadly, our results show that a single agreement algorithm can solve $k$ 
consensus problems simultaneously at a fraction of the complexity of classical 
approaches but, in general, it requires higher network connectivity.
\end{abstract}

\section{Introduction}
\label{sec:1}


\IEEEPARstart{C}{oordination}  and consensus algorithms are central to many 
network synchronization  problems, including rendezvous, distributed 
optimization, and distributed computation and sensing. 
One of the most established coordination algorithms is that of consensus, 
which can be used to compute asymptotically a \textit{common weighted average} 
of the agents' initial states -- see, for example, the representative 
works~\cite{VB-JH-AO-JT:05,RO-RM:04,WR-RWB-EMA:05}. 
This work departs from the observation that, in several applications, it 
is instead of interest to compute \textit{multiple weighted averages} of the 
initial states, each characterized by a different weighting. 
Relevant examples of this problem include distributed computation~\cite{DP:00b} 
(where agent-specific weights are used to describe heterogeneous computational 
objectives across agents), task allocation problems~\cite{HC-LB-JH:09} 
(where agent-specific weights are used to model the heterogeneous 
computational capabilities of the agents), distributed 
sensing~\cite{LZ-AA:05,FP-RC-AB-FB:10n} (where agent-specific weights describe 
heterogeneous accuracies of different sensing devices), and robotic 
formation~\cite{KO-MP-HA:15} (where agent-specific weights allow one to impose 
agent-specific configurations relative to other agents).

Mathematically, given a vector $x_0 \in \mathbb{R}^{n}$ of initial states or 
estimates -- such that each of its entries is known only locally by a 
single agent -- and a rank-$k$ matrix $W \in \mathbb{R}^{n \times n},$ whose 
rows describe the weights of the means to be computed, we say that the group 
\emph{reaches a $k$-dimensional 
agreement} when, asymptotically, the vector of agents' states converges to 
$Wx_0$.
The goal of this paper is to design distributed control protocols that enable 
the agents to reach an agreement.
\begin{figure}[t]
\centering
\subfigure[]{\includegraphics[width=.34\columnwidth]{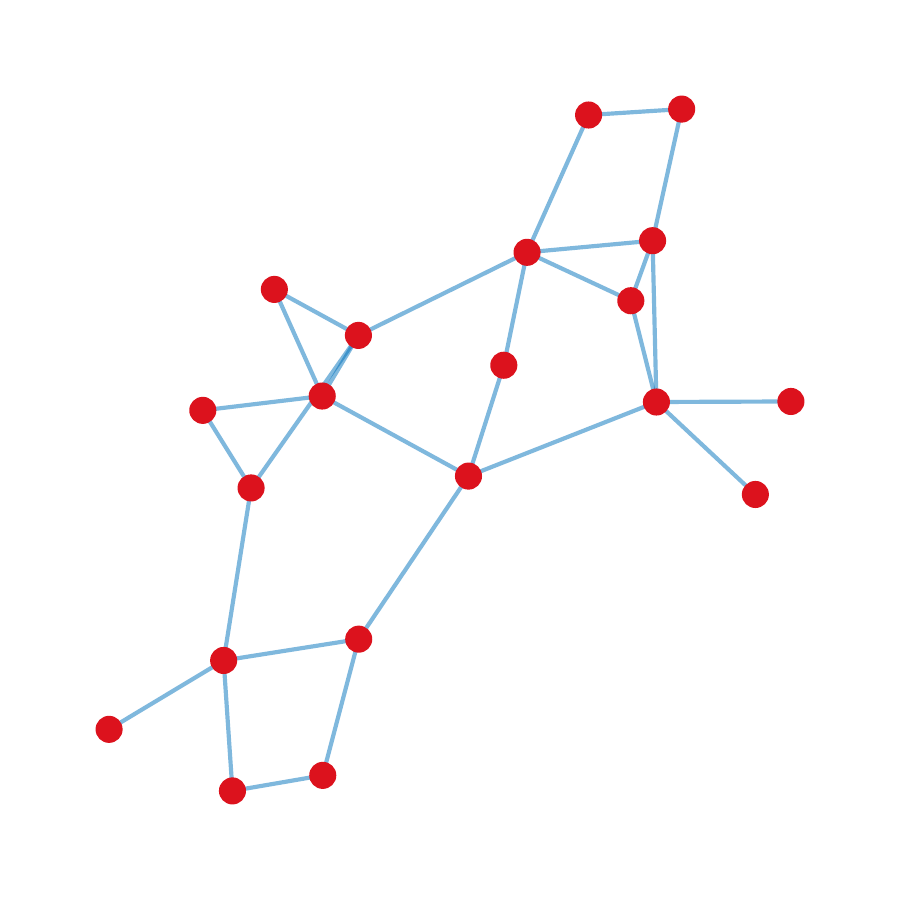}} \hspace{1.2cm}
\subfigure[]{\includegraphics[width=.34\columnwidth]{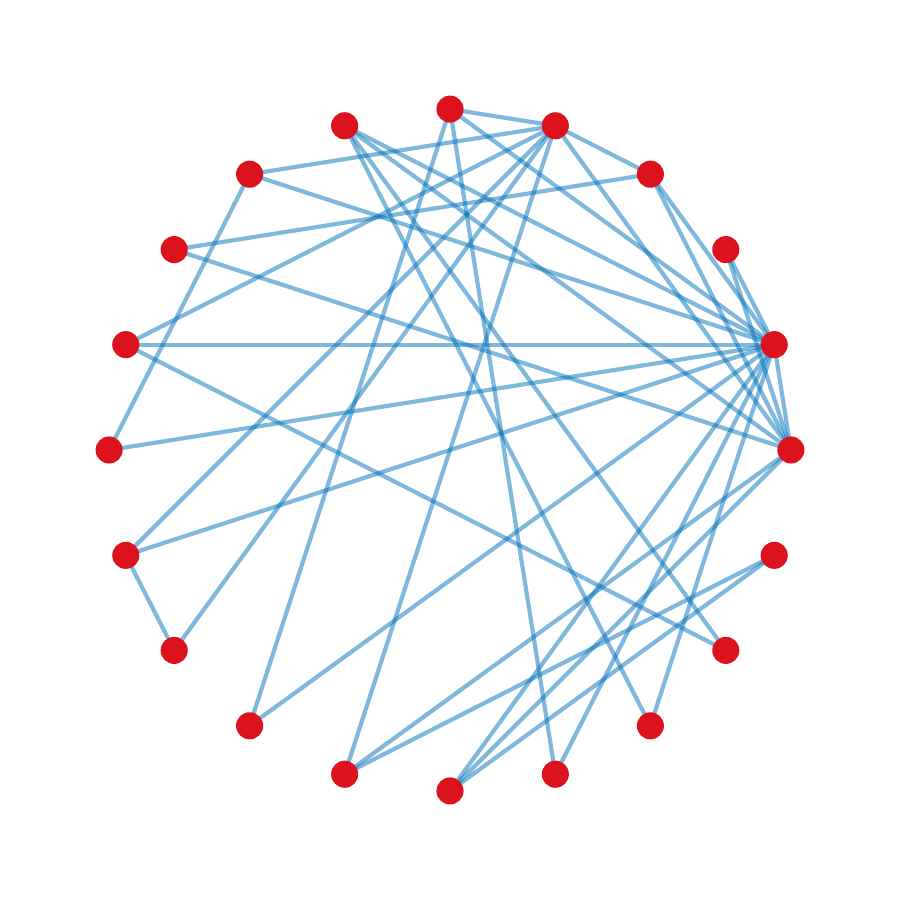}} \\
\subfigure[]{\includegraphics[width=.49\columnwidth]{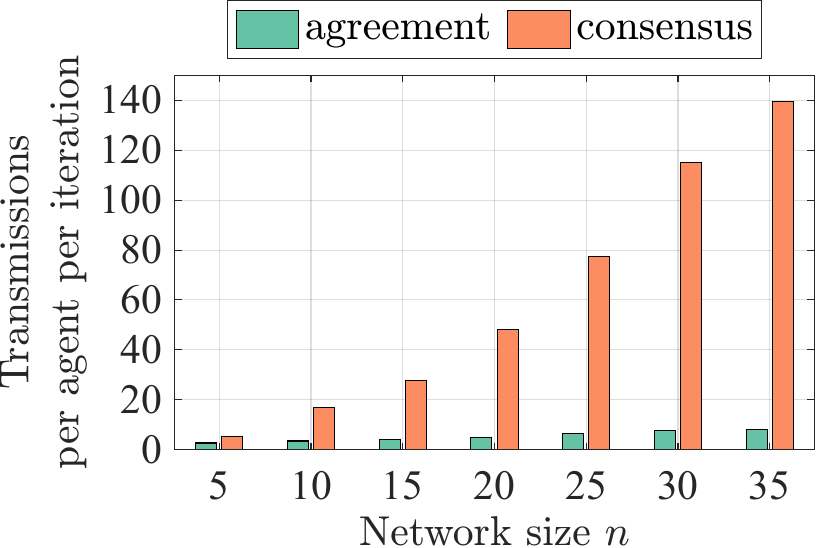}} \hfill
\subfigure[]{\includegraphics[width=.49\columnwidth]{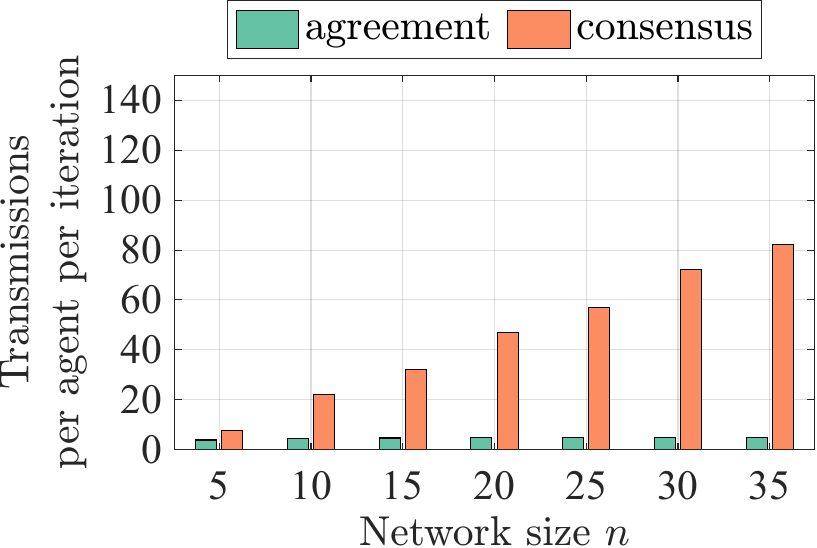}} 
\caption{Communication complexity of running $k$ consensus algorithms in 
parallel vs one $k$-dimensional agreement algorithm (proposed in this paper) to compute 
$k=\left\lfloor \frac{n}{2} \right\rfloor$ weighted average 
means of a global quantity. (a) and (c) Erd\H{o}s–R\'{e}nyi network model. (b) 
and (d) Barabasi-Albert model. Bars denote the average number of transmissions 
per iteration per~agent. See Section~\ref{sec:4-b}.}
\label{fig:communicationComplexity}
\vspace{-.5cm}
\end{figure}
A natural approach to tackle this problem consists of executing $k$ 
consensus algorithms~\cite{RO-RM:04} in parallel (see 
Fig.~\ref{fig:communicationComplexity} -- simulation details are provided in 
Section~\ref{subsec:applications}), each designed to converge to 
a specific row of $W x_0$. Unfortunately, the communication and 
computational complexities of such an approach do not scale with the network 
size (cf.~Fig.~\ref{fig:communicationComplexity}); thus, our objective here is 
to reach agreements by running \textit{a single} distributed algorithm.

{\bf \color{mblue} Related work.} 
The problem studied in this work is closely related to that of consensus. 
Consensus algorithms have been extensively studied in the literature. A  list 
of representative topics (necessarily incomplete) includes: sufficient and/or 
necessary conditions for consensus~\cite{JNT:84,AJ-JL-ASM:02,RO-RM:04,MC-ASM-BDOA:06,WR-RB:05,JMH-JNT:13}, convergence rates~\cite{AO-JNT:07,LX-SB:04}, and 
robustness investigations~\cite{DD-NL-SP-ES-WW:86}.
In contrast with constrained consensus problems~\cite{AN-AO-PP:10,FM:20} (where 
the agents' states must satisfy agent-dependent constraints during 
transients, but the desired asymptotic value is unconstrained), in our 
setting the values are instead constrained at convergence, and thus \textit{the 
agents' states do not coincide} in general.
Clustering-based consensus~\cite{SA-IS-SM-DN:17,WL-HD:09,GB-AC-ML-GM:15-cdc} 
is a closely related problem where the states of agents in the same graph 
cluster converge to identical values, while inter-cluster states can differ. 
Differently from this setting, which is obtained by using weakly-connected 
communication graphs to separate the state of different communities, here we 
are interested in cases where the asymptotic state of each agent depends on 
every other agent in the network.
To the best of our knowledge, the agreement problem proposed here has not been 
considered before in the literature. A relevant contribution is that of scaled 
consensus~\cite{SR:15}, which can be seen as a special case of the agreement 
problem studied here, obtained by letting $k=1.$ As shown shortly below, the 
extension to $k > 1$ is non-trivial as standard assumptions made for 
consensus are inadequate, see the discussion in 
Example~\ref{ex:agreementNotPossible}. 
%

{\bf \color{mblue} Contributions.}
The contribution of this work is threefold.
First, we formulate the $k$-dimensional agreement problem, and we discuss the 
fundamental limitations of linear protocols in solving this problem. 
We provide a first main result, which consists of a full characterization of the 
agreement space for linear protocols. 
Second, we provide an algebraic characterization of all agreement protocols that 
are consistent with a pre-specified communication graph.
We show how such characterization can be used to design efficient numerical 
algorithms for agreement. 
Finally, we illustrate the applicability of the framework on a regression 
problem through simulations.


\section{Preliminaries}
\label{sec:2}

{\bf \color{mblue} Notation.}
$\complex$ and $\real$ denote, respectively, the set of complex and real 
numbers.
For $x \in \complex$, $\Re(x)$ and $\Im (x)$ denote its real and imaginary 
parts, respectively.
Given $x\in \real^n, u\in\real^m$, $(x,u) \in \real^{n+m}$ denotes their 
concatenation. 
$\one_n \in \real^n$ is the vector of all ones, $I_n \in \real^{n \times n}$ is 
the identity matrix, $\zero_{n,m} \in \real^{n\times m}$ is the matrix of all 
zeros -- subscripts are dropped when dimensions are clear from the context.
For $A \in \real^{n \times n}$, 
$\spec(A) =\setdef{ \lambda \in \complex}{\det (\lambda I -A)=0}$ is  its 
spectrum, and 
$\lmax{A} = \max \setdef{\Re(\lambda)}{ \lambda \in \sigma(A)}$ is
its spectral abscissa.
For $A \in \real^{n \times m}$, $\im(A)$ and $\ker(A)$ denote its image and 
null space, respectively.
A polynomial with real coefficients
$p(\lambda)$ is \textit{stable} if all its roots have negative real part.




{\bf \color{mblue} Graph-theoretic notions.}
A \textit{digraph} is $\mc G = (\mc V, \mc E)$, where $\mc V = \until n$ 
and $\mc E \subseteq \mc V \times \mc V$ are, respectively, the set of nodes 
and edges. $(i,j) \in \mc E$ denotes a directed edge from $j$ to $i$. 
$\mc G = (\mc V, \mc E, A)$ indicates that $\mc G$ is a 
\textit{weighted digraph,} whereby the entries of the \textit{adjacency matrix} $A \in \real^{n \times n}$ describe the edge weights. 
For $A=[a_{ij}]$ to be a valid adjacency matrix, we must have: 
$(i,j) \not \in \mc E$ implies $a_{ij} = 0$.
If this holds, we say that a matrix $A$ is \textit{consistent} with $\mc G.$
A graph is \textit{complete} if there exists an edge connecting every pair 
of nodes.
A \textit{path} is a sequence of edges $(e_1, e_2, \dots )$, such that the 
initial node of each edge is the final node of the preceding one. 
The \textit{length} of a path is the number of edges contained in 
$(e_1, e_2, \dots )$.
A graph is \textit{strongly connected} if, for any 
$i, j \in \mc V$, there is a path from $i$ to $j$.
A \textit{closed path} is a path whose initial and final vertices coincide.
A closed path is a \textit{cycle} if, going along the path, one 
reaches no node, other than the initial-final node, more than once.
The \textit{length of a cycle} is equal to the number of edges in that cycle.
A set of node-disjoint cycles such that the sum of the cycle lengths is equal 
to $\ell$ is called a \textit{cycle family} of length $\ell$. We let
$\mc C_\ell(\mc G)$ denote the set of all $\ell$-long  
cycle families of $\mc G.$
See Fig.~\ref{fig:cycleFamilies} for illustration.
Since we are concerned with linear subspaces obtained by forcing certain 
entries of the matrices in $\real^{n \times n}$ to be zero, we will use the 
structural approach to system theory~\cite{KR-KJR:98}.
Given $\mc G$, we let $\mc A_{\mc G} = \{ A \in \real^{n \times n} : A \text{ 
is consistent with $\mc G$}\}$ be the vector space of all matrices 
consistent with $\mc G.$
Let $a \in \real^{\vert \mc E \vert},$ we denote by $\bA_{\mc G}(a)$ the element 
of $\mc A_{\mc G}$ parametrized by $a.$

\begin{figure}[t]
\centering \subfigure[]{\includegraphics[width=.336\columnwidth]{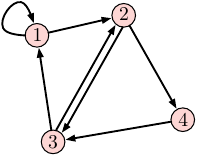}} \hspace{1cm}
\subfigure[]{\includegraphics[width=.32\columnwidth]{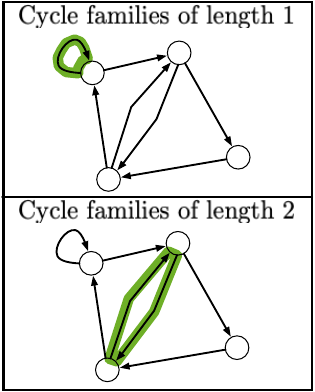}}\\
\subfigure[]{\includegraphics[width=.8\columnwidth]{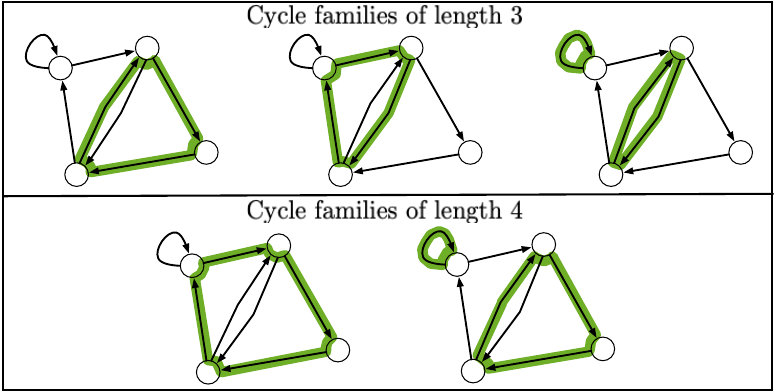}}
\caption{(a) An example of digraph $\mc G.$ (b)-(c) Illustration of all  cycle 
families of $\mc G,$ organized by length (a cycle family of length $\ell$ is a 
set of node-disjoint cycles such that the total number of edges is equal to 
$\ell$).}
\label{fig:cycleFamilies}
\vspace{-.5cm}
\end{figure}

{\bf \color{mblue} Projections and linear subspaces.}
$x, y \in \real^n$ are orthogonal 
if $x^\tsp  y = 0$; the \textit{orthogonal complement} (or \textit{orthogonal 
subspace}) of $\mc M \subset \real^n$ is 
$\mc M^\bot :=\setdef{x \in \real^n}{ x^\tsp y=0, \, \forall~y \in \mc M}$.
Given $\mc M, \mc N \subseteq \real^n$, $\mc W \subseteq \real^n$ is a 
\textit{direct sum} of $\mc M$ and $\mc N$ (denoted 
$\mc W = \mc M \oplus \mc N$) if 
$\mc M \cap \mc N = \{ 0 \}$,  and 
$\mc M + \mc N = \setdef{ u + v}{u \in \mc M, v \in \mc N} = \mc W$.
Subspaces $\mc M, \mc N \subset \real^n$ are \textit{complementary} if 
$\mc M \oplus \mc N = \real^n$. 
Matrix $\Pi \in \real^{n \times n}$ is called a \textit{projection} if 
$\Pi^2=\Pi$. 
Given complementary subspaces $\mc M, \mc N \subset \real^n$, 
for any $z \in \real^n$ there exists a unique decomposition $z=x+y$, where 
$x\in \mc M$, $y \in \mc N$.
The transformation $\Pi_{\mc M, \mc N}$, defined by 
$\Pi_{\mc M, \mc N}z := x$, is called \textit{projection onto $\mc M$ along 
$\mc N$}; $\Pi_{\mc N, \mc M},$ defined by 
$\Pi_{\mc N, \mc M}z := y,$ is called \textit{projection onto $\mc N$ along 
$\mc M$}; $x$ is the {projection of $z$ onto $\mc M$ along $\mc N$}, and 
$y$ is the {projection of $z$ onto $\mc N$ along $\mc M$}.
The projection $\Pi_{\mc M,\mc M^\bot}$ onto $\mc M$ along $\mc M^\bot$ is 
called \textit{orthogonal projection onto $\mc M$}.
Because $\mc M$ uniquely determines $\mc M^\bot$, we will denote $\Pi_{\mc M, \mc M^\bot}$ by $\Pi_{\mc M}$. 
Projections that are not orthogonal are called \textit{oblique 
projections}.

\begin{lemma}{\bf\textit{(See~\cite[Thm. 2.11 and Thm. 2.31]{AG:03})}}
\label{lem:diagonalizableProjection}
If $\Pi \in \real^{n \times n}$, $\rank \Pi = k$, is a projection, there 
exists $T \in \real^{n \times n}:$
\begin{align*}
\Pi = T \begin{bmatrix}
I_k & 0 \\ 0 & 0
\end{bmatrix}
T^\inv.
\end{align*}
Moreover, if $\Pi$ is an orthogonal projection, then $T$ can be chosen to be an 
orthogonal matrix, i.e., $T T^\tsp = I$.
\QEDB
\end{lemma}

\begin{lemma}{\bf\textit{(See~\cite[Thm. 2.26]{AG:03})
}}
\label{lem:projectionComputation}
Let $\mc M, \mc N$ be complementary subspaces and let the columns of 
$M \in \real^{n \times k}$ and $N \in \real^{n \times k}$ form a basis for 
$\mc M$ and $\mc N^\bot$, respectively. Then 
$\Pi_{\mc M, \mc N} = M(N^\tsp M)^\inv N^\tsp$.
\QEDB
\end{lemma}

We recall the following known properties~\cite[Thm.~1.60]{AG:03}:
\begin{align*}
\im(M^\tsp) &= \im(M^\pinv) = \im(M^\pinv M) = \im(M^\tsp M),\\
\ker(M) &= \im(M^\tsp)^\bot=\ker (M^\pinv M) = \im(I-M^\pinv M). 
\end{align*}
From these properties and Lemma~\ref{lem:projectionComputation}, if 
$M \in \real^{m \times n}$, then $\Pi_{\im(M)} = M M^\pinv$ and 
$\Pi_{\ker(M)} = I - M^\pinv M$, where $M^\pinv\in \real^{n \times m}$ is the Moore-Penrose inverse 
of $M$.

\section{Problem Setting}
\label{sec:3}

\subsection{Problem formulation}
Consider a set of agents $\mc V = \until n,$ each characterized by a state 
$x_i \in \real, i \in \mc V,$ and communicating through a network whose 
topology is described by a digraph $\mc G = (\mc V, \mc E)$.
We study a model where each agent exchanges its state with its 
neighbors and updates it as:
\begin{align}
\label{eq:systemDistributed}
\dot x_i &= a_{ii} x_i + \sum_{j \in \mc N_i} a_{ij} x_j, &
\forall \,\, i \in \mc V,
\end{align}
where $a_{ij} \in \real$, $(i,j) \in \mc E$, is a weighting factor, and 
$\mc N_i = \setdef{j \in \mc V\setminus \{i\}}{(i,j) \in \mc E}$ is the set of 
in-neighbors of $i.$
Setting $A= [a_{ij}], a_{ij}=0$ if $(i,j) \not \in \mc E$, and 
$x = (x_1, \dots, x_n)$, in vector form the network dynamics are:
\begin{align}
\label{eq:systemCentralized}
\dot x = A x.
\end{align}

\noindent
We say that~\eqref{eq:systemCentralized} reaches an agreement if each state 
variable converges to an agent-dependent weighted sum of the initial conditions, 
as formalized next.
\begin{definition}{\bf \textit{($k$-dimensional agreement)}}
\label{def:kDimensAgree}
Let 
$W \in \real^{n \times n}$ be such that $\rank W = k \in \naturalpos$.
We say that the update~\eqref{eq:systemCentralized} \textit{globally 
asymptotically reaches a  $k$-dimensional agreement on $W$} if, 
for any $x(0) \in \real^n$, 
\begin{align}
\label{eq:kDimensAgree}
\lim_{t \rightarrow \infty} x(t) = W x(0).
\end{align}
\QEDB\end{definition}

We discuss in Section~\ref{subsec:applications} some application scenarios for 
this notion.
Notice that agreement does not require that the agents' states coincide 
at convergence: in fact, $\lim_{t\rightarrow \infty}\norm{x_i(t)- x_j(t)}=0$ 
only holds if all rows of $W$ are identical.  
We discuss in Remark~\ref{rem:interpretationOfConsensus} how agreement 
generalizes the well-studied notion of consensus.

\begin{remark}{\bf \textit{(Relationship with consensus problems)}}
\label{rem:interpretationOfConsensus}
In the special case $k=1$, $W$ can be written as $W= v w^\tsp$ for some 
$v, w \in \real^n$. In this case, we recover the \textit{scaled 
consensus} problem~\cite{SR:15}. When, in addition, $v=\one$ and 
$w^\tsp \one=1$, we recover the \textit{consensus} 
problem, see, e.g.,~\cite{RO-RM:04}. When, $v=\one$ and $w = \frac{1}{n} \one$, 
our problem simplifies to \textit{average consensus}~\cite[Sec.~2]{RO-RM:04}. 
Notice that all state variables converge to the same quantity 
only when $k=1$ and $v = \one.$
\QEDB
\end{remark}

In line with the consensus literature, the following distinction is important.

\begin{definition}{\bf \textit{(Agreement on some weights vs on arbitrary weights)}}
\label{def:arbitrary_vs_some_weights}
Let $k \in \naturalpos$.
\begin{itemize}
\item[\textit{(i)}] The set of agents is said to be \textit{globally 
$k$-agreement reachable on some weights} if there exists $W \in \real^{n \times n},$  $\rank{W}=k,$ and $A\in \real^{n \times n}$  such 
that~\eqref{eq:systemCentralized} globally asymptotically reaches a 
$k$-dimensional agreement on $W.$
\item[\textit{(ii)}] The set of agents is said to be \textit{globally $k$-agreement 
reachable on arbitrary weights} if, for any $W\in\real^{n \times n}$  with
$\rank{W}=k,$ there exists $A$ such that~\eqref{eq:systemCentralized} globally 
asymptotically reaches a $k$-dimensional agreement on $W$.

\QEDB
\end{itemize}
\end{definition}


Extending Remark~\ref{rem:interpretationOfConsensus}, agreement reachability on 
some weights is a generalization of global consensus 
reachability~\cite{RB-VS:03}, while agreement reachability on arbitrary weights 
generalizes global \textit{average} consensus reachability~\cite{RO-JF-RM:07}. 
Importantly, whether a group of agents is agreement reachable depends on two 
main factors: (i) the choice of $k,$ and (ii) the connectivity of $\mc G$. 
We illustrate this in the following example.

\begin{example}
{\bf \textit{(Agreement on arbitrary vs on some 
weights)}}\label{ex:arbitrary_vs_some_weights}
Consider a set of agents whose communication graph is a set of isolated nodes 
with self loops (i.e., $\mc V = \until n$ and $\mc E = \{(i,i)\}_{i \in \mc V}$). 
The set of protocols~\eqref{eq:systemCentralized} compatible with this graph 
is characterized by a diagonal matrix 
$A = \diag (a_1, \dots , a_n).$ Notice that 
$\lim_{t \rightarrow \infty} x(t)= \lim_{t \rightarrow \infty}e^{At} x(0)$ 
exists if and only if $\max \{a_i\}_{i \in \mc V} \leq 0.$ 
When the latter condition holds,  
$\lim_{t \rightarrow \infty} e^{At} = \diag(d_1, \dots, d_n),$ where $d_i=0$ if 
$a_i<0$ and $d_i=1$ if $a_i=0.$ 
Hence, the agents are globally agreement reachable on some weights (precisely, 
any agreement matrix $W$ has the form $W=\diag(d_1, \dots, d_n)$). However, the 
agents are not globally  agreement reachable on arbitrary weights (in fact, 
agreement cannot be reached, for example, on any non-diagonal $W$).~
\QEDB\end{example}

With this motivation, in this work, we study the following two problems. 

\begin{problem}{\bf \textit{(Construction of communication graphs for agreement)}}
\label{probl:graph_determination}
Determine the largest class of communication graphs 
that guarantees that the set of agents is globally $k$-agreement reachable on 
arbitrary weights. 
\QEDB\end{problem}

\begin{problem}{\bf \textit{(Agreement protocol design)}}
\label{probl:matrix_determination}
Let $\mc G$ be a communication graph such that the set of agents is 
globally $k$-agreement reachable on arbitrary weights (see 
Problem~\ref{probl:graph_determination}) and let $W \in \real^{n \times n}, \rank W = k.$
Determine $A$, consistent with $\mc G,$ such that~\eqref{eq:kDimensAgree} holds 
with optimal rate of convergence.~
\QEDB\end{problem}

Problem~\ref{probl:graph_determination} is a feasibility problem: it asks to 
determine the class of graphs that support agreement protocols on 
arbitrary weights. 
Problem~\ref{probl:matrix_determination}, instead, is a protocol design 
problem. 
We conclude this section by discussing an important technical challenge related 
to designing agreement protocols.

\begin{remark}{\bf \textit{(New technical challenges with respect to consensus)}}
\label{rem:technical_challenges_wrt_consensus}
Several techniques have been proposed in the literature to design consensus 
protocols, including Laplacian-based methods~\cite{RO-RM:04}, 
distributed weight design~\cite{VS-GH-GM:14}, and centralized weight 
design~\cite{LX-SB:04}. Most of these methods rely on the assumption that the  
protocol $A$ is a non-negative matrix and on the Perron-Frobenius 
Theorem~\cite{FB:19} as the main tool for the analysis. 
Unfortunately, the Perron-Frobenius theorem can no longer be used for agreement 
problems for two reasons: (i) the entries of $W$ are possibly negative scalars 
and thus $A$ can no longer be restricted to being a non-negative matrix, and 
(ii) $A$ can no longer be restricted to being a matrix with a single dominant 
eigenvalue (as we prove in Lemma~\ref{lem:spectralPropertiesA}, shortly below). 
Hence, the agreement problem presents new theoretical challenges with respect 
to the existing literature.
\QEDB
\end{remark}

\subsection{Illustrative applications}
\label{subsec:applications}

In this section, we present some illustrative applications where 
the agreement problem emerges in practice.

{\bf \color{mblue} Distributed parallel computation of multiple functions.}
Many numerical computational tasks amount to evaluating a certain function at a 
given point~\cite{DB-JT:15}: examples include computing scalar addition, inner 
products, matrix addition and multiplication, matrix powers, finding the least 
prime factor,  etc.~\cite[Sec. 1.2.3]{DB-JT:15}.
Formally, given a function $\map{f}{\real^n}{\real}$ and a point 
$(\hat x_1, \dots, \hat x_n )$, the objective is to evaluate 
$f(\hat x_1, \dots, \hat x_n ).$ The classical approach to this problem amounts 
to designing an iterative algorithm $\dot x = g(x)$ such that 
$\lim_{t \rightarrow \infty} x(t) = f(\hat x_1, \dots, \hat x_n ).$
When such a computing task has a distributed nature~\cite{SM-PR:92}, each 
quantity $\hat x_i$ is known only by agent $i,$ and it is of interest to 
maintain $\hat x_i$ private from the rest of the network. 
In these cases, the distributed computation literature~\cite{SM-PR:92} has 
proposed the update rule $\dot x_i = g_i(x),$ to be designed such that
$\lim_{t \rightarrow \infty} x_i(t) = f(\hat x_1, \dots, \hat x_n ), \forall i.$

Consider now the problem of evaluating, in a distributed fashion, 
\textit{several} functions at a common point. 
Formally, given $\map{f_1, \dots, f_n}{\real^n}{\real}$ and 
$(\hat x_1, \dots, \hat x_n ),$ the objective is to design 
distributed protocols of the form $\dot x_i = g_i(x)$ such that:
\begin{align} \label{eq:distributed_computation}
& \lim_{t \rightarrow \infty} x_i(t) = f_i(\hat x_1, \dots, \hat x_n), &&
\forall i.
\end{align}
It is immediate to see that, when $f_i$ are linear, this is an instance of 
the agreement problem~\eqref{eq:kDimensAgree}.

{\bf \color{mblue} Constrained Kalman filtering.}
Kalman filters are widely used to estimate the states of a dynamic 
system. In constructing Kalman filters, it is often necessary to account for 
state-constrained dynamic systems; examples include camera tracking, fault 
diagnosis, chemical processes, vision-based systems, and biomedical 
systems~\cite{DS:10}.
Formally, given a dynamic system of the form $\dot x = F x + Bu + w,$ $y=Cx+e,$
subject to the state constraint $Dx=0,$ (see~\cite[eq. (10)]{DS:10}), 
the objective is that of constructing an optimal estimate $\hat x^c$ of $x$ 
given past measurements $\{y(\tau), \tau \leq t\}.$
Denoting by $\hat x^u$ the state estimate constructed using an unconstrained 
Kalman filter, a common approach to tackle the constrained problem consists of 
projecting $\hat x^u$ onto the constraint space~\cite[Sec.~2.3]{DS:10}:
\begin{align*}
\hat x^c = \arg \min_{x} ~~ \norm{x - \hat x^u}^2, &&
\text{subject to:}~  Dx = 0.
\end{align*}
The solution to this problem is 
$\hat x^c = (I- D^\tsp (D D^\tsp)^\inv D )\hat x^u;$
notice that this is an oblique projection of the vector $\hat x^u.$
To speed up the calculation, it is often of interest to parallelize the 
computation of $\hat x^c$ across a group of distributed processors. It is 
then immediate to see that the agreement  problem~\eqref{eq:kDimensAgree} 
provides a natural framework to address this problem.

\subsection{Complexity considerations}
We now illustrate how the use of classical coordination algorithms 
to solve~\eqref{eq:distributed_computation} leads to a suboptimal use of 
resources. 
Assume that functions $f_i(\cdot)$ in~\eqref{eq:distributed_computation}  
are linear, namely, 
$f_i(x) = w_i^\tsp x$, with $w_i \in \real^n$, $w_i^\tsp \one=1$, 
and that  $k$ vectors of $\{w_1, \dots , w_n\}$ are linearly independent.
It is natural to consider two approaches to solve this problem.

{\bf \color{mblue} Approach 1.}
This approach consists of running $k$ independent consensus 
algorithms~\cite{RO-JF-RM:07} in parallel, as outlined next. 
Let each agent $i$ duplicate its state $k$ times: 
$\{ x_i^{(d)} \in \real\}_{d \in \until k},$ and update the states using:
\begin{align}\label{eq:laplacian_consensus}
\dot x_i^{(d)} = \sum_{j} a_{ij}^{(d)} (x_j^{(d)}- x_i^{(d)}), && 
x_i^{(d)}(0) = \hat x_i.
\end{align}
Letting $A^{(d)} = [a_{ij}^{(d)}]$ and choosing $A^{(d)}$ such that 
$w_d^\tsp A^{(d)}=0,$ \eqref{eq:laplacian_consensus} is a Laplacian-based 
consensus algorithm~\cite[Thm. 1]{RO-JF-RM:07}; as such, 
$
\lim_{t \rightarrow \infty} x^{(d)}_i(t) =  w_d^\tsp \hat x,$
provided that $\mc G$ is strongly connected.
In words, the $d$-th state replica of each agent 
satisfies~\eqref{eq:distributed_computation}.
Unfortunately,  the spatial and communication 
complexities of this approach do not scale well with $n$ (see 
Fig.~\ref{fig:communicationComplexity}): each agent maintains $k$ replica 
state variables and, at every time step, it transmits these $k$ variables to all 
its neighbors. Thus, the per-agent spatial 
complexity is $\mathcal{O}(k)$ (since each agent maintains $k$ state copies), 
and the per-agent communication complexity\footnote{$\text{deg}(\mc G)$ denotes 
the largest among all in- and out- node degrees 
in $\mc G$.} is $\mathcal{O}(k \cdot \text{deg}(\mc G))$ and thus  
$\mathcal{O}(n \cdot \text{deg}(\mc G))$ when $k$ grows with $n$.

{\bf \color{mblue} Approach 2.}
Consider the use of protocol~\eqref{eq:systemCentralized}, designed to 
achieve~\eqref{eq:kDimensAgree} with $W=[w_1, \dots , w_n]^\tsp$. 
Deriving techniques to design such a protocol is the focus of this work, and 
will be presented shortly below (see Section~\ref{sec:5}). 
For such a protocol, the per-agent spatial complexity is $\mathcal{O}(1),$ 
since each agent maintains a single scalar state variable and the 
communication complexity is $\mathcal{O}(\text{deg}(\mc G))$. 
A comparison of the communication volumes of the two approaches is illustrated 
in Fig.~\ref{fig:communicationComplexity}.
Notice the fundamental difference between the two approaches: in~{Approach 1}, 
one computes $k$ independent quantities by \textit{running $k$ distributed averaging algorithms} while, in~{Approach 2}, one computes the $k$ independent 
quantities by \textit{running a single distributed algorithm}.


\section{Characterization of the agreement space and  fundamental limitations}
\label{sec:4}
The focus of this section is to address Problem~\ref{probl:graph_determination}.

\subsection{Algebraic characterization of agreement space}

The following result is instrumental.

\begin{lemma}{\bf \textit{(Spectral properties of agreement protocols)}}
\label{lem:spectralPropertiesA}
A set of agents with communication graph $\mc G$ is globally $k$-agreement 
reachable on some weights if and only if there exists $A\in \real^{n\times n}$ such that:
\begin{align}
\label{eq:semiConvergentA}
A \in \mc A_{\mc G}, && \text{and} &&
A = T \begin{bmatrix}
\zero_{k, k} & \zero_{k,n-k} \\ \zero_{n-k,k} & B
\end{bmatrix}
T^\inv,
\end{align}
for some nonsingular
$T \in \real^{n \times n}$ and
$B \in \real^{(n -k) \times (n-k)}$ satisfying $\lmax{B}<0$. 

Conversely, a set of agents is globally $k$-agreement reachable on arbitrary 
weights if and only if for any nonsingular $T \in \real^{n \times n},$ there 
exists $A\in \real^{n \times n}$ such that~\eqref{eq:semiConvergentA} holds.
%
\QEDB\end{lemma}

\begin{proof}
\textit{(If)} When~\eqref{eq:semiConvergentA} holds, we have that:
\begin{align*}
\lim_{t\rightarrow \infty} x(t)
= \lim_{t\rightarrow \infty} e^{At} x(0)
= \underbrace{T \begin{bmatrix}
I_k & \zero_{} \\ \zero_{} & \zero
\end{bmatrix}
T^\inv }_{:=W} x(0) 
=Wx(0).
\end{align*}

\textit{(Only if)} From~\cite[Lemma 1.7]{FB:19}, if
$\lim_{t \rightarrow \infty} e^{At}$ exists, then
$\lmax A \leq 0$; moreover, if $\lambda$ is an eigenvalue of $A$ such that
$\Re(\lambda)=0$, then  $\lambda=0$ and its algebraic and geometric 
multiplicities coincide. 
It follows that $A$ must satisfy~\eqref{eq:semiConvergentA}.
%
\end{proof}

Lemma~\ref{lem:spectralPropertiesA} provides an algebraic characterization of 
agreement protocols through~\eqref{eq:semiConvergentA}.
Next, we characterize the class of weight matrices $W$ on which an agreement 
can be reached.

\begin{proposition}{\bf \textit{(Characterization of agreement space)}} 
\label{prop:characterizationAgreementSpace}
Let $x(t)$ denote the solution of~\eqref{eq:systemCentralized} with initial 
condition $x(0)$. If $\lim_{t \rightarrow \infty} x(t):= x_\infty$ exists, 
then there exist complementary subspaces 
$\mc M, \mc N \subset \real^n$ such that $x_\infty = \Pi_{\mc M, \mc N} x(0).$
Moreover, let $\{t_1, \dots, t_k\}$ denote the first $k$ columns of $T$ 
in~\eqref{eq:semiConvergentA} and 
$\{\tau_1^\tsp, \dots, \tau_k^\tsp\}$ denote the first $k$ rows of $T^\inv.$ 
Then, 
\begin{align}
\label{eq:expressionM_N}
\mc M &= \im(\{t_1, \dots, t_k\}), &
\mc N^\bot &= \im(\{\tau_1, \dots, \tau_k\}).
\end{align}
\QEDB\end{proposition}

\begin{proof}
Assume that  $\lim_{t \rightarrow \infty} x(t)$ exists; by application of 
Lemma~\ref{lem:spectralPropertiesA}, $A$  reads as 
in~\eqref{eq:semiConvergentA}, and thus:
\begin{align*}
\lim_{t\rightarrow \infty} x(t)
= \lim_{t \rightarrow \infty} e^{A t} x(0) = 
T \begin{bmatrix} I_k & \zero \\ \zero & \zero \end{bmatrix} T^\inv 
x(0).
\end{align*}
Notice that $W^2=W,$ and thus $W$ is a projection matrix; by application 
of~\cite[Thm 2.6]{AG:03}, we conclude that there exists subspaces 
$\mc M, \mc N,$ (given by $\mc M = \im(W)$ and $\mc N = \ker(W))$ such that 
$W = \Pi_{\mc M, \mc N}.$

To prove the second part of the claim, let $T$ and $T^\inv$ be: 
\begin{align*}
T &=
\begin{bmatrix}
t_1 & \cdots & t_n 
\end{bmatrix},  &
T^\inv &=
\begin{bmatrix}
\tau_1 & \cdots & \tau_n 
\end{bmatrix}^\tsp, &
\end{align*}
where $t_i, \tau_i\in \real^n, i \in \until n.$ Then, 
\begin{align*}
\lim_{t\rightarrow \infty} x(t)
&= 
\begin{bmatrix} t_1 & \cdots & t_k \end{bmatrix}
\begin{bmatrix} \tau_1 & \cdots & \tau_k \end{bmatrix}^\tsp 
x(0).
\end{align*}
From Lemma~\ref{lem:projectionComputation}, it follows that the subspaces 
$\mc M$ and $\mc N$ are  given by~\eqref{eq:expressionM_N}.
%
%
\end{proof}



Proposition~\ref{prop:characterizationAgreementSpace} is a 
fundamental limitation-type result: it shows that if $\dot x = Ax$ converges, 
then the asymptotic value is some oblique projection of the initial conditions 
$x(0)$. 
In turn, this implies that linear protocols can agree only on weight matrices
$W$ that are oblique projections.


\begin{remark}{\bf \textit{(Geometric reinterpretation of consensus 
algorithms)}}
In the case of consensus, the group of agents is known to converge to 
$\one w^\tsp x(0)$, where $w$ is the left eigenvector of $A$  that 
satisfies $w^\tsp \one =1$ (see Remark~\ref{rem:interpretationOfConsensus}).
Proposition~\ref{prop:characterizationAgreementSpace} allows us to give a 
geometric interpretation of the consensus value: 
$\one w^\tsp x(0) = \Pi_{\mc M, \mc N} x(0)$ is the oblique projection of 
$x(0)$ onto $\mc M = \im(\one)$ along $\mc N = \im(w)^\bot.$
In the case of average consensus, the convergence value
(given by $\frac{1}{n}\one \one^\tsp x(0)$) is the orthogonal projection of 
$x(0)$ onto $\mc M = \im(\one)$.~
\QEDB\end{remark}

Motivated by the conclusions in 
Proposition~\ref{prop:characterizationAgreementSpace}, in what follows we make
the following assumption.

\begin{assumption}{\bf \textit{(Matrix of weights is a projection)}}
\label{as:projectionWeights}
The matrix of weights $W$ is a projection. Namely, $W \in \real^{n \times n},$
$W^2=W,$ and $\rank W =k$. 
\QEDB\end{assumption}

Notice that, given two complementary subspaces $\mc M, \mc N,$  a matrix 
$W$ that satisfies Assumption~\ref{as:projectionWeights} can be computed as:
\begin{align*}
W = M(N^\tsp M)^\inv N^\tsp,
\end{align*}
where 
$M \in \real^{n \times k}$ and $N \in \real^{n \times k}$ form a basis for 
$\mc M$ and $\mc N^\bot$, respectively (see 
Lemma~\ref{lem:projectionComputation}).
Notice that the agreement space corresponding to this choice of $W$ is 
$\Pi_{\mc M, \mc N}$.

We are now ready to prove the following.

\begin{proposition}{\bf \textit{(Existence of agreement algorithms over complete digraphs)}} 
\label{prop:existenceWeights}
Let $\mc M, \mc N \subset \real^n$ be complementary subspaces and
$\mc G$ the complete graph. 
There exists $A \in \real^{n \times n}$, consistent with 
$\mc G,$ such that the 
iterates~\eqref{eq:systemCentralized} satisfy 
$\lim_{t \rightarrow \infty} x(t) = \Pi_{\mc M, \mc N} x(0)$.
\QEDB\end{proposition}

\begin{proof}
For any pair of complementary subspaces $\mc M, \mc N,$ 
\cite[Thm. 2.26]{AG:03} guarantees the existence of an oblique projection 
matrix $\Pi_{\mc M, \mc N}$.
Moreover, by Lemma~\ref{lem:diagonalizableProjection}, there exists invertible 
$T_\Pi \in \real^{n \times n}$ such that $\Pi_{\mc M, \mc N}$ can be decomposed 
as
\begin{align*}
\Pi_{\mc M, \mc N} = T_\Pi \begin{bmatrix}
I_k & \zero \\ \zero & \zero
\end{bmatrix}
T_\Pi^\inv,
\end{align*}
where $k = \dim(\mc M)$. 
The statement follows by choosing $A$ as in~\eqref{eq:semiConvergentA} with 
$T=T_\Pi$ and by noting that, with this choice, 
$\lim_{t\rightarrow \infty} e^{At}x(0) = \Pi_{\mc M, \mc N} x(0)$.
%
\end{proof}

Proposition~\ref{prop:existenceWeights} provides a preliminary answer to 
Problem 1: if the communication graph is complete, a set of agents is globally 
$k$-agreement reachable on arbitrary weights, $\forall \ k \in \naturalpos$.
The proof is constructive, and it provides a way to derive agreement 
protocols --  see Algorithm~\ref{alg:constructA}.
We remark that, for some special choices of $\mc M, \mc N,$ one or more entries 
of $A$ may be identically zero (notice that such $A$ remain consistent with our 
definition of adjacency matrix for complete graphs -- see 
Section~\ref{sec:2}); in these cases, the protocol $A$ could also be 
implemented over a non-complete graph. However, in 
the general case, $A$ is nonsparse.

\begin{algorithm}
\caption{Construction of agreement matrix $A$}
\begin{algorithmic} \label{alg:constructA}
\REQUIRE $M \in \real^{n \times k}$ whose columns form a basis for $\mc M$
\REQUIRE $N \in \real^{n \times k}$ whose columns form a basis for $\mc N^\bot$
\STATE $\Pi_{\mc M, \mc N} \gets  M(N^\tsp M)^\inv N^\tsp$;
\STATE Determine $T$ such that 
$\Pi_{\mc M, \mc N} = T \begin{bmatrix} I_k & \zero \\ \zero & \zero \end{bmatrix} T^\inv$;
\STATE Choose $B \in \real^{(n -k) \times (n-k)}$ such that $\lmax{B}<0$;
\RETURN $ A = T \begin{bmatrix} \zero_k & \zero \\ \zero & B \end{bmatrix} T^\inv$;
\end{algorithmic}
\end{algorithm}
\vspace{-.5cm}

\subsection{Structural necessary conditions for agreement}
\label{sec:4-b}

While Proposition~\ref{prop:existenceWeights} shows that complete graphs 
can reach an agreement on arbitrary weights, it remains unclear whether 
this property also holds for graphs with weaker connectivity. 
We begin by showing that strong connectivity\footnote{Recall that strong 
connectivity is necessary and sufficient for global \textit{average} 
consensus reachability~\cite{RO-JF-RM:07}.} is necessary but not sufficient 
for agreement reachability on arbitrary weights.

\begin{lemma}{\bf \textit{(Necessity of strong connectivity)}}
\label{lem:strongConnectivity}
A set of agents is {globally $k$-agreement reachable on arbitrary weights} only 
if $\mc G$ is strongly connected. 
\QEDB\end{lemma}

\begin{proof}
When $\mc G$ is not strongly connected, for all $A$ consistent with $\mc G,$ at 
least one of the entries of $\lim_{t \rightarrow \infty }e^{At}$ is identically 
zero (this follows from $e^{At} = \sum_{i=0}^\infty \frac{A^i t^i}{i !}$  
and~\cite[Cor~4.5]{FB:19}). In this case, since $W=\lim_{t \rightarrow \infty }e^{At}$, an agreement cannot be reached on every $W$ such that $w_{ij}\neq 0 \ \forall \ i,j.$~
\end{proof}


\begin{example}{\bf \textit{(Strong connectivity is not sufficient for 
agreement on arbitrary weights)}}\label{ex:agreementNotPossible}
Assume that a network of $n=3$ agents is interested in agreeing on a space with 
$k=2$ by using a non-complete communication graph~$\mc G$.
By using Lemma~\ref{lem:spectralPropertiesA}, the agents are agreement reachable 
on arbitrary weights only if:
\begin{align}
\label{eq:auxChoiceA}
A = 
\underbrace{
\begin{bmatrix}
t_1 & t_2 & t_3
\end{bmatrix} }_{=T}
\begin{bmatrix}
0 & 0 & 0\\ 0 & 0 & 0\\ 0 & 0 & \beta
\end{bmatrix}
\underbrace{
\begin{bmatrix}
\tau_1 & \tau_2 & \tau_3
\end{bmatrix}^\tsp }_{=T^\inv} 
= \beta t_3 \tau_3^\tsp,
\end{align}
for some $\beta$ such that $\Re (\beta) <0$ and some $T\in \real^{3\times 3}$.
By~\eqref{eq:auxChoiceA}, $A$ must be a rank-one matrix and, since $\mc G$ is 
not complete, at least one of the entries of $A$ must be identically zero. 
These two properties imply that at least one of the rows or columns of $A$ must 
be identically zero, and thus that $\mc G$ cannot be strongly connected.
Since $\mc G$ is not strongly connected, by~\cite[Cor~4.5]{FB:19} at least one 
of the rows or columns of $W=\lim_{t \rightarrow \infty }e^{At}$ must be 
identically zero. 
In summary, we have found that the agents are globally $2$-agreement reachable 
on arbitrary weights only if $\mc G$ is the complete graph. 
\QEDB\end{example}

We will thus make the following necessary assumption.

\begin{assumption}{\bf \textit{(Strong connectivity)}}
\label{as:strongConnectivity}
The communication digraph $\mc G$ is strongly connected.
\QEDB\end{assumption}

\section{Agreement algorithms over sparse digraphs}
\label{sec:5}

While~\eqref{eq:semiConvergentA} gives a full characterization of 
agreement protocols and can be used to design agreement algorithms over 
complete graphs (cf.~Algorithm~\ref{alg:constructA}), it remains unclear how to 
design agreement protocols when $\mc G$ is not complete. This is the 
focus of this section. 
We will often use the following decomposition for $W$ (see 
Assumption~\ref{as:projectionWeights} and 
Lemma~\ref{lem:diagonalizableProjection}):
\begin{align}
\label{eq:decompositionW}
W &= 
T
\begin{bmatrix}
I_k & \zero \\ \zero & \zero
\end{bmatrix}
T^\inv,
\end{align}
where $T\in\real^{n \times n}$ is invertible.
Moreover, we  will use:
\begin{align}\label{eq:decompositionW_2}
T &=
\begin{bmatrix}
t_1 & \cdots & t_n 
\end{bmatrix},  &
T^\inv &=
\begin{bmatrix}
\tau_1 & \cdots & \tau_n 
\end{bmatrix}^\tsp, &
\end{align}
where $t_i, \tau_i \in \real^n,$ 
$i \in \until n$ (notice that $\tau_i^\tsp t_j=1$ if 
$i=j$ and $\tau_i^\tsp t_j=0$ otherwise).

\subsection{Algebraic conditions for sparse digraphs}
We will use a graph-theoretic interpretation of characteristic 
polynomials~\cite{KR:94}, which we recall next. 
Recall that $\mc C_\ell(\mc G)$ denotes the set of all $\ell$-long  
cycle families of $\mc G$ (see Section~\ref{sec:2}).
\begin{lemma}{\bf\textit{(~\!\!\cite[Thm.~1]{KR:94})}}
\label{lem:reinschkeCoefficientsCharPoly}
Let $\mc G$ be a digraph, let $A \in \mc A_{\mc G}$, and 
$\det (\lambda I - A) = 
\lambda^n + p_1 \lambda^{n-1} + \dots + p_{n-1} \lambda + p_n$ be its 
characteristic polynomial.
Then, for all $p_\ell, \ell \in \until n:$ 
\begin{align*}
p_\ell= \sum_{\xi \in \mc C_\ell(\mc G)} (\negat 1)^{d(\xi)} \prod_{(i,j) \in \xi} a_{ij},
\end{align*}
where $d(\xi)$ is the number of cycles in cycle family $\xi$.
\QEDB\end{lemma}

The lemma provides a graph-theoretic description of the characteristic 
polynomial: it shows that the $\ell$-th coefficient of $\det (\lambda I - A)$ 
is a sum of terms; each summand is the product of edges in a cycle family of 
length $\ell$ of $\mc G$. 

\begin{example}
Consider the digraph in \figurename~\ref{fig:cycleFamilies}(a). We have:
\begin{align*}
\bA_{\mc G}(a) = \begin{bmatrix}
a_{11} & 0 & a_{13} & 0\\
a_{21} & 0 & a_{23} & 0\\
0 & a_{32} & 0 & a_{34}\\
0 & a_{42} & 0 & 0\\
\end{bmatrix},
\end{align*}
and we refer to \figurename~\ref{fig:cycleFamilies}(b)-(c) for an illustration 
of all cycle families of this graph.
Lemma~\ref{lem:reinschkeCoefficientsCharPoly} yields:
\begin{align*}
p_1 &= -a_{11}, & 
p_{3} &= -a_{13}a_{21}a_{32} + a_{11}a_{23}a_{32} - a_{23}a_{42}a_{34},\\
p_2 &= -a_{23}a_{32}  &
p_4 &= -a_{13}a_{21}a_{42}a_{34} + a_{11}a_{23}a_{34}a_{42}.
\end{align*}
Notice that each summand in $p_\ell$ is the product of weights in a cycle 
family of the corresponding size (cf.~\figurename~\ref{fig:cycleFamilies}(b)-(c)).
\QEDB\end{example}

The following result is one of the main contributions of this paper.
Recall that for $a \in \real^{\vert \mc E \vert},$ $\bA_{\mc G}(a)$ is the  
matrix consistent with $\mc G$ whose entries are parametrized by $a.$

\begin{theorem}{\bf \textit{(Algebraic characterization of sparse agreement matrices)}}
\label{thm:algebraicCharacterization}
Let Assumptions~\ref{as:projectionWeights}-\ref{as:strongConnectivity} hold.
$\dot x = \bA_{\mc G}(a)x$ globally asymptotically reaches a 
$k$-dimensional agreement on $W$ if and only if the following hold simultaneously:
\begin{itemize}
\item[\textit{(i)}] $\bA_{\mc G} (a)  t_i = 0, \quad 
\tau_i^\tsp \bA_{\mc G} (a) = 0,  \quad\quad  \forall i \in \until k$;
\vspace{.2cm}
\item[\textit{(ii)}] The polynomial 
$\lambda^{n-k-1} + p_1 \lambda ^{n-k-2} + \dots + p_{n-k-1}$, whose coefficients are defined as
\begin{align*}
p_\ell &= \sum_{\xi \in \mc C_\ell(\mc G)} (\negat 1)^{d(\xi)} \prod_{(i,j) \in \xi} a_{ij},
& \ell \in \until{n-k},
\end{align*}
is stable.
\QEDB
\end{itemize}
\end{theorem}

\begin{proof}
\textit{(If)}  Let $A$ be any matrix that satisfies \textit{(i)-(ii)}. If $A$ 
is diagonalizable, then, by letting $T = (t_1, \cdots, t_n)$ 
be the matrix of its  right eigenvectors and 
$(T^\inv)^\tsp = ( \tau_1, \cdots, \tau_n)$ 
be the matrix of its  left eigenvectors, we 
conclude that $A$ satisfies~\eqref{eq:semiConvergentA} and thus the linear 
update reaches an agreement on $W$. If $A$ is not diagonalizable, let $T$ be a
similarity transformation such that $T^\inv A T$ is in Jordan normal form:
\begin{align*}
T^\inv A T = 
\begin{bmatrix}
J_{\lambda_1} \\
& J_{\lambda_2}\\
&& \ddots \\
&&& J_{\lambda_n}
\end{bmatrix},
J_{\lambda_i} = 
\begin{bmatrix}
\lambda_1 & 1 \\
& \ddots & \ddots \\
&& \lambda_1 \\
\end{bmatrix},
\end{align*}
From \textit{(i)} we conclude that $\lambda=0$ is an eigenvalue with algebraic 
multiplicity $k$, moreover, since the vectors $t_i$ are linearly independent 
(see~\eqref{eq:decompositionW}), we conclude that its geometric multiplicity is 
also equal to $k$, and thus all Jordan blocks associated with $\lambda=0$ have 
dimension $1$. Namely, $J_{\lambda_1} = \dots = J_{\lambda_k} =0$.
By combining this with \textit{(ii)}, we conclude that the characteristic 
polynomial of $A$ is 
\begin{align*}
\det(\lambda I - A) 
& = \lambda^n + p_1 \lambda ^{n-1} + \dots + p_{n-k-1} \lambda^{k-1},
\end{align*}
and, since by assumption such polynomial is stable, we conclude that all 
remaining eigenvalues $\{ \lambda_{k+1}, \dots , \lambda_{n} \}$ of $A$ satisfy 
$\Re(\lambda_i)<0$.
Since all Jordan blocks associated with $\lambda=0$ have dimension $1$ and 
all the remaining eigenvalues of $A$ are stable, we conclude that $A$ admits 
the representation~\eqref{eq:semiConvergentA} and thus the linear 
update reaches an agreement.

\textit{(Only if)} We will prove this claim by showing 
that~\eqref{eq:semiConvergentA} implies \textit{(i)-(ii)}.
To prove that \textit{(i)} holds, we rewrite~\eqref{eq:semiConvergentA} as
\begin{align*}
T^\inv A T = \begin{bmatrix}
\zero_{} & \zero_{} \\ \zero_{} & B
\end{bmatrix},
\end{align*}
and, by taking the first $k$ columns of the above identity we conclude 
$A t_i =0$, $i \in \until k$, thus showing that \textit{(i)} holds.
To prove that \textit{(ii)} holds, notice that~\eqref{eq:semiConvergentA} 
implies that the characteristic polynomial of $A$ is a stable polynomial with 
$k$ roots at zero. Namely, 
\begin{align*}
\det(\lambda I - A) &= \lambda^k(\lambda-\lambda_1)(\lambda-\lambda_2) \cdots (\lambda-\lambda_{n-k})\\
& = \lambda^n + p_1 \lambda ^{n-1} + \dots + p_{n-k-1} \lambda^{k-1},
\end{align*}
where $\Re(\lambda_i)<0, i \in \until{n-k},$ and $p_j, j \in \until{n-k-1},$ 
are nonzero real coefficients.
The statement \textit{(ii)} thus follows by applying the graph-theoretic 
interpretation of the coefficients of the characteristic polynomial in 
Lemma~\ref{lem:reinschkeCoefficientsCharPoly}.
\end{proof}

Theorem~\ref{thm:algebraicCharacterization} provides an algebraic 
characterization of agreement protocols over sparse digraphs. The result
is remarkable as it can be used to design sparse agreement protocols as 
follows. Given $\mc G$ and $W$, we interpret $a$ as well as 
$p_1, \dots p_{n-k}$ as free parameters or unknowns; then, \textit{(i)-(ii)} 
define a system of equations (precisely, $2nk$ linear equations and $n-k$ 
multilinear polynomial equations) in these unknowns.
Any solution to this system of equations -- yielding a stable characteristic 
polynomial -- gives an agreement protocol on $W$ consistent with $\mc G.$
Notice that the solvability of these equations is not guaranteed in general, 
but it can be assessed via standard techniques, as discussed in the following 
remark.
%

\begin{remark}{\bf \textit{(Determining solutions to systems of polynomial equations)}}
A powerful technique for determining solutions to systems of polynomial 
equations uses the tool of Gr\"obner bases, as applied using 
Buchberger's algorithm~\cite{DC-JL-DO:13}. The technique relies on transforming 
the system of equations into a canonical form, expressed in terms of a 
Gr\"obner basis, for which it is then easier to determine a solution. 
We refer to \cite{DC-JL-DO:13,MK-LR:00} for a complete discussion.
Furthermore, existence of solutions can be assessed using Hilbert's 
Nullstellensatz theorem~\cite{MK-LR:00}.
In short, the theorem guarantees that a system of polynomial equations has no 
solution if and only if its Gr\"obner basis is $\{1\}$. In this sense, the 
Gr\"obner basis 
method provides an easy way to check solvability of \textit{(i)-(ii)}. 
We also note that the computational complexity of solving systems of polynomial 
equations via Gr{\"o}bner bases is exponential~\cite{MK-LR:00}.~
\QEDB\end{remark}

\subsection{Fast distributed agreement algorithms}
We next tackle Problem~2.
The freedom in the choice of $p_1, \dots , p_{n-k}$ in 
Theorem~\ref{thm:algebraicCharacterization} suggests that a certain graph may 
admit multiple consistent agreement protocols.
We will now leverage such freedom to seek protocols with optimal rate of 
convergence. Problem~2 can be made formal as follows:
\begin{align}\label{opt:informalMaxRate}
\underset{A}{\text{min}} ~~~&  r(A)  \nonumber\\
\text{s.t.} ~~~ & A \in \mc A_{\mc G}, \quad \lim_{t \rightarrow \infty } e^{A t} = W. 
\end{align}
In~\eqref{opt:informalMaxRate}, $\map{r}{\real^{n \times n}}{\real}$ is a 
function that measures the rate of convergence of $e^{At}$.
By Lemma~\ref{lem:spectralPropertiesA}, the 
optimization~\eqref{opt:informalMaxRate} is feasible if and only 
if~\eqref{eq:semiConvergentA} holds with $T$ given 
by~\eqref{eq:decompositionW}.

When the optimization problem~\eqref{opt:informalMaxRate} is feasible,
it is natural to consider two possible choices for the cost function 
$r(\cdot)$, 
motivated by the size of $\norm{e^{At}}$ as a function of time.
The first limiting case is $t\rightarrow \infty$. In this case, we consider the 
following asymptotic measure of convergence motivated 
by~\cite[Ch.~14]{LT-ME:05}:
\begin{align}
\label{eq:spectralAbscissa}
r_\infty(A) := \lim_{t \rightarrow \infty } t^\inv \log \norm{e^{At}} = \lmax{A},
\end{align}
where $\lmax{A}$ is the spectral abscissa of $A$ (see Section~\ref{sec:2}).
The second limiting case is $t \rightarrow 0$. In this case:
\scalebox{.94}{\parbox{\linewidth}{
\begin{align}
\label{eq:numericalAbscissa}
r_0(A) := \frac{d}{dt} \norm{e^{At}} \bigg\vert_{t=0} 
&= \lim_{t \downarrow 0} t^\inv \log \norm{e^{At}} 
= \lmax{\frac{A + A^\tsp}{2}},
\end{align}
}}

\noindent where $ \lmax{\frac{A + A^\tsp}{2}}$ is the numerical abscissa of $A$~\cite{LT-ME:05}.

We have the following result.
\begin{proposition}{\bf \textit{(Fast agreement problem)}}\label{prop:spectral_abscissa}
Let Assumptions~\ref{as:projectionWeights}-\ref{as:strongConnectivity} hold.
Assume that the optimization problem~\eqref{opt:informalMaxRate} is feasible.
Any solution to the following  optimization problem:
\begin{align}\label{opt:formalSpectralAbscissa}
\min_{a \in \real^{\vert \mc E \vert}} ~~&  r(\bA_\mc G(a)) \\
\text{s.t.} ~~~ & \bA_\mc G(a) t_i=0, ~~\tau_i^\tsp \bA_\mc G(a)=0,  \quad\quad  i \in \until k,\nonumber
\end{align}
where $t_i, \tau_i$ are as in~\eqref{eq:decompositionW_2}, 
is also a solution of~\eqref{opt:informalMaxRate}.
\QEDB\end{proposition}

\begin{proof}
Since the optimization problem~\eqref{opt:informalMaxRate} is feasible, 
condition \textit{(i)} of Theorem~\ref{thm:algebraicCharacterization} 
guarantees that~\eqref{opt:formalSpectralAbscissa} is also feasible and that 
$\lmax{\bA_\mc G(a)}<0$.
Let $a^*$ denote a solution of~\eqref{opt:formalSpectralAbscissa}, and let 
$A = \bA_\mc G (a^*)$. By construction, we have $A \in \mc A_\mc G$, while the 
two constraints in~\eqref{opt:formalSpectralAbscissa}, {together with 
$\lmax{A}<0$ (which is guaranteed by feasibility),} guarantee that 
$\lim_{t \rightarrow \infty } e^{A t} = W$, which shows that $a^*$ is a 
feasible point for~\eqref{opt:informalMaxRate}. 
The claim thus follows by noting that the cost functions of~\eqref{opt:informalMaxRate} and that 
of~\eqref{opt:formalSpectralAbscissa} coincide. 
\end{proof}

Proposition~\ref{prop:spectral_abscissa} allows us to 
recast~\eqref{opt:informalMaxRate} as a finite-dimensional search over the 
parameters $a \in \real^{\vert \mc E \vert}$.
We remark that~\eqref{opt:formalSpectralAbscissa} with the numerical abscissa 
formulation~\eqref{eq:numericalAbscissa} is a convex optimization 
problem~\cite{JB-MO:01}, while with the spectral abscissa 
formulation~\eqref{eq:spectralAbscissa}, finding solutions may be 
computationally burdensome because the objective function may be non-convex 
(or even non-Lipschitz~\cite{JB-MO:01}).

\section{Applications and numerical validation}
\label{sec:8}

Consider a distributed estimation problem characterized by a regression 
model of the form $y=H\theta + w$, where $H \in \real^{n \times k}, n>k,$ 
$\theta \in \real^k$ is an unknown parameter, and $w \in \real^{n}$ models 
noise.
We assume that each agent $i$ can sense the $i$-th entry of vector $y$, 
denoted by $y_i$, and the group of agents is interested in cooperatively solving 
the  regression problem:
\begin{align}\label{eq:regression}
\sbs{\theta}{ls} := \arg \min_\theta \norm{H\theta-y}.
\end{align}
It is well-known that $\sbs{\theta}{ls}$ is given by
$
\sbs{\theta}{ls} = (H^\tsp H)^\inv H^\tsp y,
$
provided that $H^\tsp H$ is invertible.
Thus, the vector to be computed by the agents (de-noised measurements) is:
\begin{align*}
\hat y = H \sbs{\theta}{ls} = H (H^\tsp H)^\inv H^\tsp y,
\end{align*}
which is the orthogonal projection of $y$ onto $\im(H)$. 
For figure illustration purposes, we consider the case 
$n=50$ (meaning $n=50$ agents or sensors in the network) and $k=2$ (meaning the
sensor measurements is interpolated using a line). 
We computed an agreement protocol using the optimization
problem~\eqref{eq:numericalAbscissa}-\eqref{opt:formalSpectralAbscissa} with 
weights $W = H (H^\tsp H)^\inv H^\tsp$ and implemented on a circulant 
graph~\cite{RO-RM:04}, where each agent communicates with its $4$ nearest
neighbors.
Fig.~\ref{fig:regression}(top) shows the sampling points $y$ and asymptotic 
estimates $\hat y$ in comparison with the true regression model. As expected, 
\eqref{eq:systemDistributed}  converges to the data points corresponding to the 
Mean Square Error Estimator.
Fig.~\ref{fig:regression}(bottom) shows the trajectories of the agents' states. 
Notice that the agreement state is a $50$-dimensional vector constrained to a 
$2$-dimensional subspace.

\begin{figure}[t]
\centering
\subfigure{\includegraphics[width=.99\columnwidth]{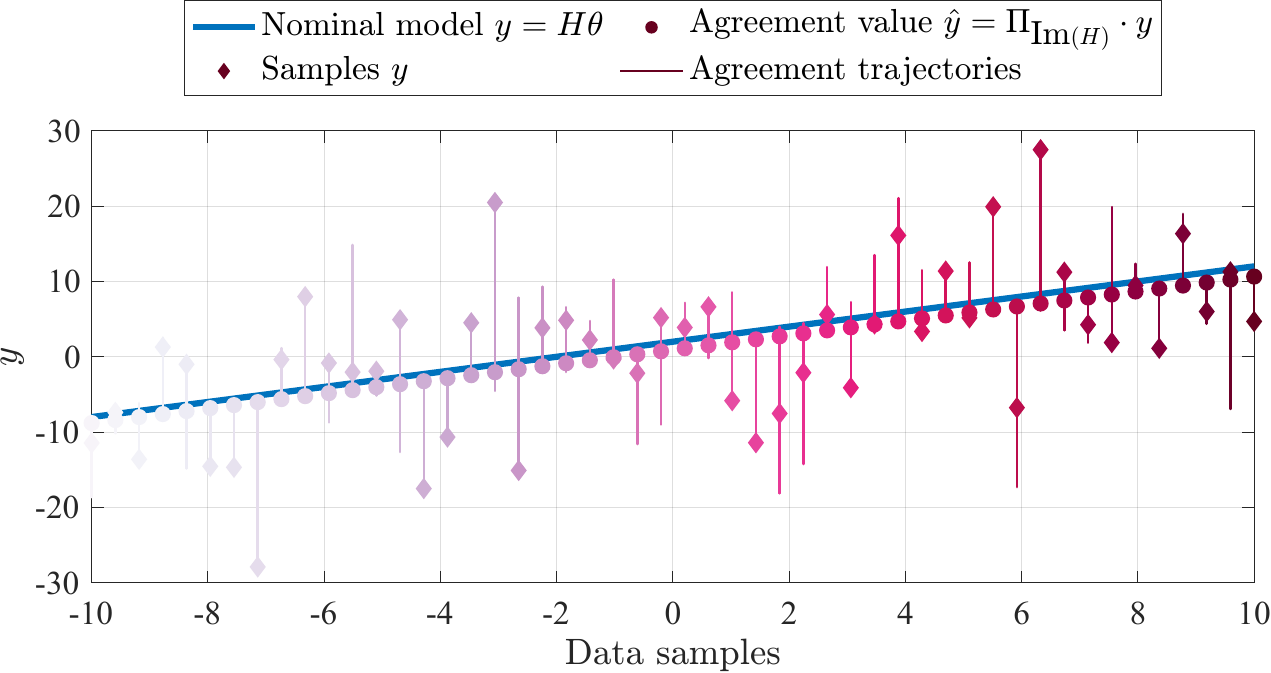}} \\
\subfigure{\includegraphics[width=\columnwidth]{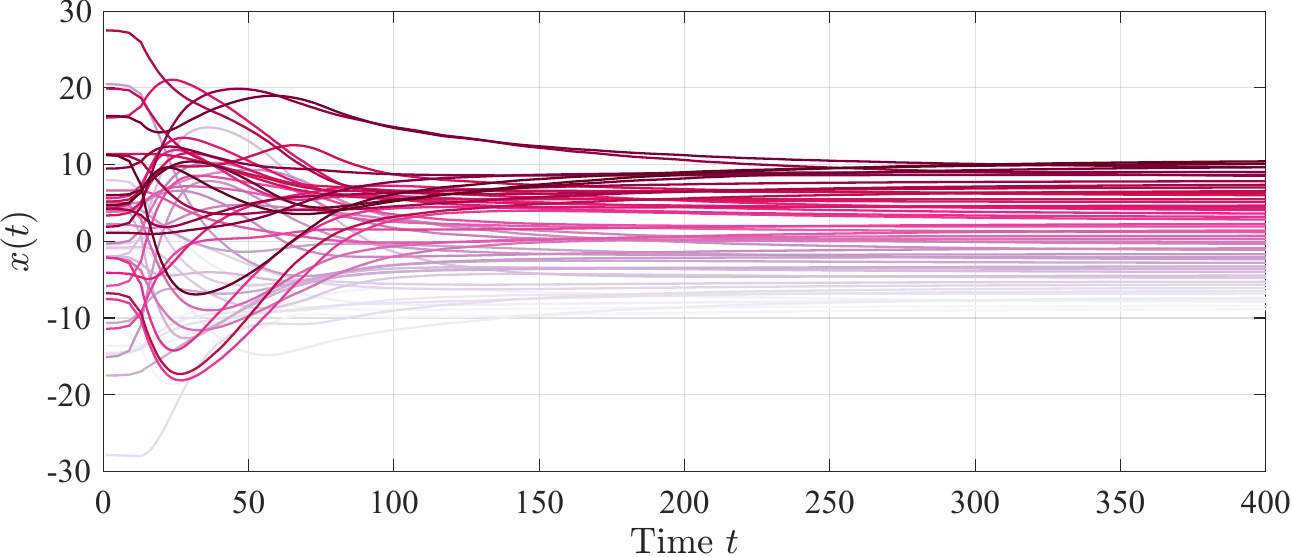}} 
\caption{Application of the agreement problem to solve a regression 
problem. Each agent can measure a sample $y_i$ (represented by diamond 
markers) and cooperatively computes the projection of $\hat y_i$ onto the range 
of the regression matrix $H.$ (Top figure) continuous lines illustrate the time
evolution of the states of~\eqref{eq:systemDistributed}. (Bottom figure) Time 
evolution of the trajectories of~\eqref{eq:systemDistributed}. Notice that the 
agents' states do not converge to the same value.}
\label{fig:regression}
\vspace{-.5cm}
\end{figure}


\section{Conclusions}
\label{sec:9}
We studied the $k$-dimensional agreement problem, whereby a group of 
agents seeks to compute $k$ independent weighted means of the agents' initial 
states.
We provided algebraic conditions to check the feasibility of the problem and 
algorithms to design such protocols. Our results show that agreement 
protocols can compute several weighted means of the agents' initial conditions
at a fraction of the complexity of existing consensus algorithms.
This work opens the opportunity for multiple directions of future research; 
among them, we highlight the derivation of graph-theoretic conditions to solve 
Problem 1, the design of agreement protocols in a distributed way, the use of 
nonlinear dynamics, and the synthesis of distributed  protocols to solve 
optimization problems over networks.

\bibliographystyle{IEEEtran}
\bibliography{BIB/alias,BIB/combined,BIB/main_GB,BIB/GB,BIB/additional_bib}

\end{document}